\documentclass{amsart}
\usepackage{amscd,amsfonts,amssymb,amsmath}
\usepackage[margin=3.8 cm]{geometry}

\usepackage{amsmath}

\newtheorem{theorem}{Theorem}[section]
\newtheorem{lemma}[theorem]{Lemma}
\theoremstyle{definition}

\newtheorem{prop}[theorem]{Proposition}
\theoremstyle{remark}

\numberwithin{equation}{subsection}
\theoremstyle{plain}

\newtheorem{cor}{Corollary}

\newtheorem{problem}{Problem}

\newtheorem*{theorema}{Theorem \ref{nn2}}

\def\Z{\mathbb Z}
\def \N{{\rm N}}
\def \n {{\widetilde{\rm N}} }
\def \l {l_{\mathcal{P}}}
\def \pw{{\rm pw}}

\newcommand{\secref}[1]{Section~\ref{#1}}
\newcommand{\thmref}[1]{Theorem~\ref{#1}}
\newcommand{\lemref}[1]{Lemma~\ref{#1}}

\newcommand{\corref}[1]{Corollary~\ref{#1}}
\newcommand{\probref}[1]{Problem~\ref{#1}}
\newcommand{\eqnref}[1]{~{\textrm(\ref{#1})}}
\numberwithin{equation}{section}

\begin{document}
\title[ON PALINDROMIC WIDTH]
{ON PALINDROMIC WIDTH OF  CERTAIN EXTENSIONS AND QUOTIENTS OF FREE NILPOTENT GROUPS}

\author{Valeriy ~G.~Bardakov}
\address{Sobolev Institute of Mathematics, Novosibirsk State University, Novisibirsk 630090  }
\address{and} 
 \address{Laboratory of Quantum Topology, Chelyabinsk State University, Brat’evKashirinykh street 129, Chelyabinsk 454001, Russia}
\email{bardakov@math.nsc.ru}
\author{Krishnendu Gongopadhyay}
\address{Department of Mathematical Sciences, Indian Institute of Science Education and Research (IISER) Mohali,
Knowledge City, Sector 81, S.A.S. Nagar, P.O. Manauli 140306, India}
\email{krishnendu@iisermohali.ac.in, krishnendug@gmail.com}

\thanks{The authors gratefully acknowledge the support of the Indo-Russian DST-RFBR project grant DST/INT/RFBR/P-137}
\thanks{Bardakov is partially supported by Laboratory of Quantum Topology of Chelyabinsk State University (Russian Federation government grant 14.Z50.31.0020) } 

\date{March 13, 2014}


\begin{abstract}
In \cite{BG} the authors provided a bound for the palindromic widths of free abelian-by-nilpotent group $AN_n$
of rank $n$ and free nilpotent group
${\rm N}_{n,r}$ of rank $n$ and step $r$.
In the present paper we study palindromic widths of groups $\widetilde{AN}_n$ and $\widetilde{\rm N}_{n,r}$.
We denote by $\widetilde{G}_n = G_n / \langle \langle x_1^2, \ldots, x_n^2 \rangle \rangle$ the quotient of
group $G_n = \langle x_1, \ldots, x_n \rangle$, which is free
in some variety by  the normal subgroup generated by
$x_1^2, \ldots, x_n^2$.
We prove that the palindromic width of the quotient $\widetilde{AN}_n$ is finite and bounded by $3n$.
We also prove that the palindromic width of the quotient
$\widetilde{\rm N}_{n, 2}$ is precisely $2(n-1)$. As a corollary to this result, 
we improve the lower bound of the palindromic width of $\N_{n, r}$.  We also improve the bound of the palindromic width of a free metabelian group. We prove that the palindromic width of a free metabelian group of rank $n$ is at most  $4n-1$. 
\end{abstract}
\maketitle
\section{Introduction}

Let $X$ be a set of generators of a group $G$.  A reduced word $w$ in the
alphabet $X^{\pm 1}$ is called a {\it palindrome} if $w$ reads the same
left-to-right and right-to-left. An element $g$ of $G$ is called a
\emph{palindrome} if $g$ can be represented by some word $w$ that is a
palindrome in the alphabet $X^{\pm 1}$. We denote the set of all palindromes in
$G$ by $\mathcal{P} = \mathcal{P}(X)$. Evidently,  the set $\mathcal{P}$
generates $G$. Then any element $g \in G$ is a product of palindromes
$$
g = p_1 p_2 \ldots p_k.
$$
The minimal $k$ with this property is called the \emph{palindromic length} of $g$ and
is denoted by $l_{\mathcal{P}}(g)$.  The \emph{palindromic width} of $G$ is given by
$$
{\rm pw}(G,X) = {\rm wid}(G, \mathcal{P}) = \underset{g \in G}{\sup} \ l_{\mathcal{P}}(g).
$$
When there is no confusion about the underlying set of generators $X$, the palindromic width with respect to $X$ is simply denoted by $\pw(G)$. 
In analogy with commutator width of groups,
it is an interesting  problem  to study palindromic width of groups.
Palindromes in groups have been investigated by several people and it has already been
useful in studying various aspects of combinatorial group theory and geometry, for example see
\cite{BST, BT}, \cite{col}--\cite{kr},  \cite{P}.

\medskip Our primary aim in this article is to investigate the
following problem:  Let $G$ be a group that is free in some variety $\mathcal G$ of groups.
Let $G=F_n(\mathcal G)$ has rank $n$. Let $X=\{x_1, \ldots, x_n\}$ be a basis of $G$.
Define a group $\widetilde G$ that is the quotient group of $G$ by the relations $x_i^2=1$, $i=1, \ldots, n$. That is
$$\widetilde G=G/\langle \langle x_1^2, \ldots, x_n^2 \rangle \rangle,$$
where $\langle \langle x_1^2, \ldots, x_n^2 \rangle \rangle$ denotes the normal closure of the set
$\{x_1^2, \ldots, x_n^2\}$. Let $Y=\{y_1, \ldots, y_n\}$ be the image of the generating set
$X = \{x_1, \ldots, x_n\}$ in $\widetilde G$. Then $\widetilde G = \langle Y \rangle$.
\begin{problem}\label{pr1}
What is the value of  $\pw(\widetilde G, Y)$?
\end{problem}

 The main motivation for us to ask the \probref{pr1} was to improve the lower bound of $\pw(\N_{n,r})$. There is an epimorphism from $G$ to $\tilde G$ and hence it follows that {$\pw(\tilde G,Y) \leq \pw(G,X)$}, see \lemref{onto} below. So obtaining precise value of $\pw(\tilde G, Y)$ would give a lower bound to the palindromic width of $G$. 

   \medskip Let ${\rm N}_{n, r}$ be the free nilpotent group of rank $n$ and of step $r$. Let $X$ be a basis of ${\rm N}_{n, r}$.  Investigation to obtain precise value of the palindromic width of ${\rm N}_{n, r}$ has been initiated by the authors in \cite{BG}. The authors have provided a bound for the palindromic width of ${\rm N}_{n, r}$: for $n>1$ and $r \geq 2, \ n \leq  {\rm pw}({\rm N}_{n, r}, X) \leq 3n$, see \cite[Theorem 1.1]{BG}. For ${\rm N}_{n,1}$ and ${\rm N}_{2,2}$ precise values of the palindromic widths were obtained. Further the upper bound was improved for $r=2$: ${\rm pw}({\rm N}_{n, 2}, X) \leq 3(n-1)$.  In this paper we further improve the lower bound of $\pw(\N_{n, r}, X)$.  The key step in this improvement is the calculation of $\pw(\n_{n,2}, Y)$.  We prove the following. 
\begin{theorem}\label{nn2}
\begin{enumerate}
\item For $n, r \geq 2$, $n \leq {\rm pw}(\n_{n,r}, Y) \leq 2n$.
\item The palindromic width of $\n_{n, 2}$ (with respect to the set of generators $Y$) is $2(n-1)$.
\end{enumerate}
\end{theorem}
As a corollary to this theorem, we have the improvement of lower bound of $\pw(\N_{n,r},X)$. 
\begin{cor}\label{cor1}
Let $\N_{n, r}$ be the $r$-step free nilpotent group of rank $n \geq 2$.  Then
\begin{enumerate}
\item For $n>1$ and $r \geq 3$, $2(n-1) \leq \pw(\N_{n,r},X) \leq 3n$.
\item  $2(n-1) \leq \pw(\N_{n,2},X)\leq 3(n-1)$. \end{enumerate}
\end{cor}

For the group $\n_{3,2}$, we have proved something more.
\begin{prop}\label{n:32}
{\it Let $z_{ij}$ denote the commutator $[y_i, y_j]$ in $\n_{3,2}$. In $\n_{3,2}$ the only element that can
not be expressed as a product of three palindromes is $z_{21} z_{31} z_{32}$.
Moreover, $\l(z_{21} z_{31} z_{32})=4$.}
\end{prop}

We recall that a group $G$  is \emph{boundedly generated } if there exist $a_1, \ldots, a_k \in G$ such that every element can be expressed as $a_1^{n_1} \ldots a_k^{n_k}$ for some integers $n_1, \ldots, n_k$, for eg. see  \cite{ck, s}. Note that there are free metableian groups those are not boundedly generated. The authors \cite{BG} and, Riley and Sale \cite{rs} have independently established the finiteness of palindromic width of free metabelian groups using different techniques. The authors actually proved a stronger result: any free abelian-by-nilpotent group of rank $n$ has palindromic width at most $5n$, cf. \cite[Section 3.4]{BG}. In particular, it follows that ${\rm pw}(M_n, X) \leq 5n$. We further improve this bound in this article.
\begin{theorem}\label{math}
Let $M_n$ be a free metabelian group of rank $n$.  Then
$$\pw(M_n, X)\leq 4n-1.$$
\end{theorem}
We also investigate \probref{pr1} for free abelian-by-nilpotent groups. This lead to further examples of finitely generated groups which are not boundedly generated but have finite palindromic width.
\begin{theorem}\label{abn}
(i) Let $AN_n$ be a free abelian-by-nilpotent group of rank $n$. Then for $n \geq 2$,
$n \leq \pw(\widetilde{AN}_n, Y) \leq 3n$.  

\medskip (ii) Further, for a free metabelian group $M_n$ of rank $n$ we have, $n \leq \pw(\widetilde{M}_n, Y) \leq 3n-1$. 
\end{theorem}

\medskip After reviewing some preliminary results in \secref{prel}, we prove Theorem \ref{nn2} in \secref{nil}. The \secref{n32}  is devoted to the proof of Proposition \ref{n:32}. In \secref{abns}, we 
investigate palindromic width of $\widetilde G$
for $G$ free abelian-by-nilpotent.  We prove  \thmref{math} and \thmref{abn}  in this section.

\subsection*{Acknowledgements}
We thank Elisabeth Fink, Andrew Sale and Tim Riley for their comments on this article. 


\section{Preliminaries} \label{prel}
\subsection{Palindromes in Groups}
In this Section we collect results that will be useful for us.

\begin{lemma}\label{onto} \cite{BG}
Let $G = \langle X \rangle$ and $H = \langle Y \rangle$ be two groups,
$\mathcal{P}(X)$ is the set of palindromes in the alphabet $X^{\pm 1},$
$\mathcal{P}(Y)$ is the set of palindromes in the alphabet $Y^{\pm 1}.$ If
$\varphi : G \longrightarrow H$ be an epimorphism such that
$\varphi(X) = Y$,  then
$$
{\rm pw}(H,Y) \leq {\rm pw}(G,X).
$$
\end{lemma}
\begin{lemma} \label{l:32} { \cite[Lemma 2.4]{BG}}
Let $G=\langle X \rangle$ be a group generated by a set $X$. Then the following
hold.
\begin{enumerate}
\item{ If $p$ is a palindrome, then for $m$ in $\mathbb Z$, $p^m$ is also a palindrome.}

\item{ Any element in $G$ which is conjugate to a product of $n$ palindromes, $n
\geq 1$,  is a product of $n$ palindromes if $n$ is even, and of $n+1$ palindromes
if $n$ is odd. }

\item{Any commutator of the type  $[u, p],$  where $p$ is a palindrome is
a product of $3$ palindromes.
Any element $[u, x^{\alpha}] x^{\beta}$, $x \in X$, $\alpha, \beta \in
\mathbb{Z},$ is a product of $3$ palindromes.}

\item{In $G$ any commutator of the type  $[u, p q],$  where $p, q$ are
palindromes is a product of $4$ palindromes.
Any element $[u, p x^{\alpha}] x^{\beta}$, $x \in X$, $\alpha, \beta \in
\mathbb{Z},$ is a product of $4$ palindromes.}
\end{enumerate}
\end{lemma}


\begin{lemma}\label{reh} \cite[Lemma 3]{MR}
Let $A$ be a normal subgroup of $G = \langle x_1, x_2, \ldots, x_n \rangle$. If $A$ is abelian or $A$ lies in the second center of $G $, then every element of $[ A, G]$ has the form
$$[u_1, x_1] \, [u_2, x_2] \, \ldots \, [u_n, x_n] \hbox{ for some }u_i \in A.$$
\end{lemma}

\subsection{Free Nilpotent Groups} Let ${\rm N}_{n,r}$ be the free $r$-step nilpotent group of rank $n$ with a basis $x_1, \ldots,
x_n$.
 For example, when $r=1$,
${\rm N}_{n,1}$ is simply the free abelian group generated by $x_1, \ldots, x_n$, so
every element of ${\rm N}_{n,1}$ can be presented uniquely as
 $$g=x_1^{\alpha_1} \ldots x_n^{\alpha_n}$$
 for some integers $\alpha_1, \ldots, \alpha_n$. For $r=2$,  every element $g \in {\rm N}_{n,2}$ has the
form
 \begin{equation}\label{fr1}g=\prod_{i=1}^n x_i^{\alpha_i} \cdot \prod_{1 \leq j < i
\leq n} [x_i, x_j]^{\beta_{ij}} \end{equation}
 for some integers $\alpha_i$  and $\beta_{ij}$, where $[x_i, x_j] = x_i^{-1} x_j^{-1} x_i x_j$
 are basic commutators (see  \cite[Chapter 5]{MKS}).

\medskip For the free nilpotent group ${\rm N}_{n, r} $,  let ${\rm N}_{n, r} '$ be its commutator subgroup. We note the following lemmas that will be used later.
\begin{lemma} \label{ar1} \cite{AR1, MR}
Any element  $g$ in the commutator subgroup ${\rm N}_{n,r}'$ can be
represented in the form
$$
g = [u_1, x_1] \, [u_2, x_2] \, \ldots \, [u_n, x_n],~~~u_i \in {\rm N}_{n,r}.
$$
\end{lemma}
\begin{lemma}\label{c1}  \cite[Corollary 1]{BG} The following
inequalities hold:
$$
{\rm pw}({\rm N}_{n,1}) \leq {\rm pw}({\rm N}_{n,2}) \leq {\rm pw}({\rm N}_{n,3}) \leq \ldots
$$
\end{lemma}

\subsubsection{Normal Forms}\label{nof} In \cite[Lemma 3.3]{BG} we have found normal form for palindromes in
$\N_{2,2}$. In this subsection we will find normal form for palindromes in $\N_{n,2}$, $n \geq 2$.

\begin{lemma}\label{nf}
Every palindrome $p \in \N_{n,2}$ has the form
$$p=x_1^{\alpha_1} \ldots x_{i-1}^{\alpha_{i-1}} x_{i+1}^{\alpha_{i+1}} \ldots x_n^{\alpha_n}
\cdot x_i^{\alpha_0} \cdot  x_n^{\alpha_n} \ldots x_{i+1}^{\alpha_{i+1}}x_{i-1}^{\alpha_{i-1}}
\ldots x_1^{\alpha_1},$$
for some integers $\alpha_0, \alpha_1 \ldots, \alpha_n$.
\end{lemma}
\begin{proof}
Any palindrome in $\N_{n, 2}$ by definition is equal to
$$p=ux_i^{\alpha} \bar u, ~ \alpha \in \Z, ~ u \in \N_{n,2},$$
where
$$
u = a_1^{\alpha_1} a_2^{\alpha_2} \ldots a_k^{\alpha_k},~~a_i \in A, \alpha_i
\in \mathbb{Z}
$$
is a word and
$$
\overline{u} = a_k^{\alpha_k} a_{k-1}^{\alpha_{k-1}} \ldots a_1^{\alpha_1}
$$
is its reverse word.

Let $$u=\prod_{j=1}^n x_j^{\alpha_j} \prod_{1 \leq l <k \leq n}  z_{kl}^{\beta_{kl}}, ~\mbox{where}~ z_{kl}=[x_k, x_l], ~ \alpha_j, \beta_{kl} \in \Z$$
be an arbitrary element in $\N_{n, 2}$.
Then
$$
\bar u = \prod_{1 \leq l <k \leq n}  {\bar z}_{kl}^{\beta_{kl}} \prod_{j=0}^{n-1} x_{n-j}^{\alpha_{n-j}},
$$
where $\bar u, ~ \bar z$ are the reverse words of $u, ~ z$ respectively.
Since ${\bar z}_{kl} =x_l x_k x_l^{-1} x_k^{-1} = [x_l^{-1}, x_k^{-1}]=[x_l, x_k]=[x_k, x_l]^{-1} = z_{kl}^{-1}$,
we have
$$
\bar u= \prod_{1 \leq l <k \leq n}  {z}_{kl}^{-\beta_{kl}} \prod_{j=0}^{n-1} x_{n-j}^{\alpha_{n-j}}.
$$
Using the rules
$$
x_i^{\alpha_i} x_j^{\alpha_j} = x_j^{\alpha_j} x_i^{\alpha_i} [x_j, x_i]^{-\alpha_i \alpha_j},
 ~ 1 \leq i < j \leq n,
$$
$$
x_j^{\alpha_j} x_i^{\alpha_i} = x_i^{\alpha_i} x_j^{\alpha_j} [x_j, x_i]^{\alpha_i \alpha_j},
 ~ 1 \leq i < j \leq n,
$$
we can remove elements $x_i^{\alpha_i}$ in the product $ux_i^{\alpha} \bar u$ to the center and cancel the
commutators, see \cite[Section 3.2]{BG} and we see that
$$
p=\prod_{\substack{j=1\\ j\not=i}}^n x_j^{\alpha_j} \cdot x_i^{\alpha_0} \cdot
\prod_{\substack{j=0\\j\not=n-i}}^{n-1} x_{n-j}^{\alpha_{n-j}},~~~\alpha_0 = \alpha + 2 \alpha_i.
$$
This completes the proof.
\end{proof}

Using this Lemma we find the normal form for palindromes in $\N_{n,2}$.

\begin{prop}\label{pnf}
{\it  There are $n$ different types of palindromes in $\N_{n,2}$ that can be written in the following
normal forms.}
$$ p_1=x_1^{\alpha_1} x_2^{2 \alpha_2} \ldots x_n^{2 \alpha_n} \prod_{t=2}^n z_{t1}^{\alpha_t \alpha_1}
 \prod_{2 \leq l < k \leq n} ~z_{kl}^{2 \alpha_k \alpha_l},$$
$$\vdots$$
$$
p_j=x_1^{2 \alpha_1} \ldots x_{j-1}^{2 \alpha_{j-1}} x_j^{\alpha_j} x_{j+1}^{2 \alpha_{j+1}}
\ldots x_n^{2 \alpha_n}\underset{k, l \neq j} { \prod_{1 \leq l < k \leq n}} z_{kl}^{2 \alpha_k \alpha_l}
\cdot \prod_{t=1}^{j-1} ~ z_{jt}^{\alpha_j \alpha_t} \cdot \prod_{s=j+1}^n z_{sj}^{\alpha_s \alpha_j},~~~ 2 \leq j \leq n-1.
$$
$$\vdots$$
$$
p_n=x_1^{2 \alpha_1} \ldots x_{n-1}^{2 \alpha_{n-1}} x_n^{\alpha_n} \prod_{1 \leq l < k \leq n-1}
z_{kl}^{2 \alpha_l \alpha_k} \cdot  \prod_{t=1}^{n-1} z_{nt}^{\alpha_n \alpha_t}.$$
\end{prop}

\begin{proof}
We give a proof for the case $p_j$, $2 \leq j \leq n-1$. The cases $p_1$ and $p_n$ are similar and simpler.
By Lemma \ref{nf} we have
$$
p_j = x_1^{\alpha_1} \ldots x_{j-1}^{\alpha_{j-1}} x_{j+1}^{\alpha_{j+1}} \ldots x_n^{\alpha_n}
\cdot x_j^{\alpha_j} \cdot  x_n^{\alpha_n} \ldots x_{j+1}^{\alpha_{j+1}}x_{j-1}^{\alpha_{j-1}}
\ldots x_1^{\alpha_1}.
$$
Remove the element $x_j^{\alpha_j}$ to the left
$$
p_j = x_1^{\alpha_1} \ldots x_{j-1}^{\alpha_{j-1}} x_j^{\alpha_j} x_{j+1}^{\alpha_{j+1}} \ldots x_n^{2 \alpha_n}
 \ldots x_{j+1}^{\alpha_{j+1}} x_{j-1}^{\alpha_{j-1}}
\ldots x_1^{\alpha_1} \cdot z_{nj}^{\alpha_n \alpha_j} z_{n-1,j}^{\alpha_{n-1} \alpha_j} \ldots
z_{j+1,j}^{\alpha_{j+1} \alpha_j}.
$$
Removing the right occurrence of $x_1^{\alpha_1}$ to the left
$$
p_j = x_1^{2 \alpha_1}  \ldots x_{j-1}^{\alpha_{j-1}} x_j^{\alpha_j} x_{j+1}^{\alpha_{j+1}}
\ldots x_n^{2 \alpha_n}
 \ldots x_{j+1}^{\alpha_{j+1}} x_{j-1}^{\alpha_{j-1}}
\ldots x_2^{\alpha_2} \cdot \prod_{t=j+1}^{n} z_{tj}^{\alpha_t \alpha_j} \cdot z_{j1}^{\alpha_j \alpha_1}
\cdot
\underset{t \neq j} { \prod_{t=2}^n} z_{t1}^{2 \alpha_t \alpha_1}.
$$
By  the similar manner removing the right occurrence of $x_2^{\alpha_2},\ldots, x_{j-1}^{\alpha_{j-1}}$
to the left we get
$$
p_j = x_1^{2 \alpha_1} \ldots x_{j-1}^{2 \alpha_{j-1}} x_j^{\alpha_j} \left(
x_{j+1}^{\alpha_{j+1}} \ldots x_n^{2 \alpha_n}
 \ldots x_{j+1}^{\alpha_{j+1}} \right)  c$$
where 
 $$
 c=\prod_{t=j+1}^{n} z_{tj}^{\alpha_t \alpha_j} \cdot  \prod_{l=1}^{j-1} z_{jl}^{\alpha_j \alpha_l}
\cdot
\underset{t \neq j} { \prod_{t=2}^n} z_{t1}^{2 \alpha_t \alpha_1}
\underset{t \neq j} { \prod_{t=3}^n} z_{t2}^{2 \alpha_t \alpha_2}
\ldots
{ \prod_{t=j+1}^n} z_{t,j-1}^{2 \alpha_t \alpha_{j-1}}.
$$
Represent the expression in the brackets in the normal form
$$
x_{j+1}^{\alpha_{j+1}} \ldots x_n^{2 \alpha_n}
 \ldots x_{j+1}^{\alpha_{j+1}} = x_{j+1}^{2\alpha_{j+1}} x_{j+2}^{2\alpha_{j+2}} \ldots x_n^{2 \alpha_n}
\cdot
{ \prod_{j+1 \leq l < k \leq n}} z_{kl}^{2 \alpha_k \alpha_l}.
$$
Hence
$$
p_j = x_1^{2 \alpha_1}  \ldots x_{j-1}^{2 \alpha_{j-1}} x_j^{\alpha_j}
x_{j+1}^{2\alpha_{j+1}} \ldots x_n^{2 \alpha_n}
 \cdot
 \prod_{t=j+1}^{n} z_{tj}^{\alpha_t \alpha_j} \cdot  \prod_{l=1}^{j-1} z_{jl}^{\alpha_j \alpha_l}
\cdot
\underset{k \neq j} {\underset{l<k\leq n} {\prod_{1\leq l \leq j-1}}} z_{kl}^{2 \alpha_k \alpha_l}
\cdot   \prod_{j+1\leq l <k \leq n} z_{kl}^{2 \alpha_k \alpha_l}.
$$
We see that this expression is equal to the needed formula.
\end{proof}

\section{Proof of \thmref{nn2}}\label{nil}
Now consider the group $\widetilde{\rm N}_{n,r}=\N_{n,r}/\langle  \langle x_1^2, \cdots, x_n^2 \rangle \rangle$.
Let $ z_{ij}=[y_i,  y_j]$, $1 \leq j < i \leq n$.
For $i=1, \ldots, n$,  $y_i= y_i^{-1}$ in $\widetilde{\rm N}_{n, r}$.

\medskip In this Section we prove:.

\begin{theorema}
\begin{enumerate}
\item For $n, r \geq 2$, $n \leq {\rm pw}(\widetilde{\rm N}_{n,r}) \leq 2n$.
\item The palindromic width of $\widetilde{\rm N}_{n, 2}$ is $2(n-1)$.
\end{enumerate}
\end{theorema}

The first part of this Theorem follows from the following assertion.

\begin{lemma}\label{n2}
For $n, r \geq 2$, $n \leq {\rm pw}(\widetilde{\rm N}_{n,r}) \leq 2n$.\end{lemma}
\begin{proof} We proved in \cite{BG} that ${\rm pw}(\widetilde{\rm N}_{n,1})=n$. Hence, the left hand side
inequality holds.

We claim that any element $g$ in $\widetilde{\rm N}_{n,r}$, $r \geq 2$, can be represented in the form
$$g=[u_1, y_1]y_1^{\alpha_1} [u_2, y_2] y_2^{\alpha_2} \ldots [u_n, y_n] \, y_n^{\alpha_n}, ~
\alpha_i \in \{0, 1\}.$$
We shall use induction on $r$.  If $r=2$ and $g \in \widetilde{\rm N}_{n,2}$ then it follows from  \hbox{\lemref{ar1}} that
$$
g = y_1^{\alpha_1}\, y_2^{\alpha_2}\, \ldots \, y_n^{\alpha_n} \, [u_1, y_1] \, [u_2, y_2] \, \ldots \,
[u_n, y_n],~~~\alpha_i \in \{0, 1\},~~u_i \in \widetilde{\rm N}_{n,2}.
$$
But the commutators $[u_i, y_i],$ $i =1, 2, \ldots, n,$ lie in the center of $\widetilde{\rm N}_{n,2}$. Hence
$$
g = [u_1, y_1] \, y_1^{\alpha_1}\, [u_2, y_2]\, y_2^{\alpha_2}\,  \ldots \,
[u_n, y_n] \, y_n^{\alpha_n}, ~ \alpha_i \in \{0, 1\}.
$$
has the required form. Let the result holds for groups  $\widetilde{\rm N}_{n,r}$.  We claim that the result also
holds for $\widetilde{\rm N}_{n,r+1}$.  Let
$\Gamma = \gamma_{r+1}(\widetilde{\rm N}_{n,r+1}) = [\gamma_r(\widetilde{\rm N}_{n,r+1}), \widetilde{\rm N}_{n,r+1}]$.
Then an element $g$ of $\widetilde{\rm N}_{n,r+1}$
has the form
$$
g = [u_1, y_1] \, y_1^{\alpha_1}\, [u_2, y_2]\, y_2^{\alpha_2}\,  \ldots \,
[u_n, y_n] \, y_n^{\alpha_n} \, d
$$
for some $d \in \Gamma$, $\alpha_i \in \{0,1\}$.  It follows from \lemref{reh},
$$
d = [a_1, y_1] \, [a_2, y_2] \, \ldots \, [a_n, y_n],~~\mbox{for some}~ a_i \in \gamma_r(\widetilde{\rm N}_{n,r+1}).
$$
Since all $[a_i, y_i]$ lie in the center of $\widetilde{\rm N}_{n,r+1}$, hence
$$
g =
[u_1, y_1] \, [a_1, y_1] \, y_1^{\alpha_1}\, [u_2, y_2]\, [a_2, y_2] \, y_2^{\alpha_2}\,  \ldots \,
[u_n, y_n] \, [a_1, y_1]  \, y_n^{\alpha_n} 
$$
$$
= [u_1 a_1, y_1] \,  y_1^{\alpha_1}\, [u_2 a_2, y_2]\,  y_2^{\alpha_2}\,  \ldots \,
[u_n a_n, y_n]   \, y_n^{\alpha_n}
$$
has the required form.

By  \lemref{cmm}, any element $[u_i, y_i] \, y_i^{\alpha_i}$ is a product of 2 palindromes and $g$ is a
product of $2n$ palindromes.
\end{proof}

 The following lemma follows by imitating the proof of \cite[Lemma 3.6]{BG} and using \lemref{cmm}.

\begin{lemma}\label{n3}
Any element in $\widetilde{\rm N}_{n,2},$ $n \geq 2$ is a product of at most $2(n-1)$ palindromes.
\end{lemma}

To prove that the palindromic width of $\widetilde{\rm N}_{n,2}$ is at least $2(n-1)$, it is enough to find some
element in $\widetilde{\rm N}_{n,2}$ that can not be represented as a product of less than $2(n-1)$ palindromes.
To do this we introduce some notations. Let
$$Bas_n=\{ z_{kl} \ | \ 1 \leq l < k \leq n\}$$
is the set of all basis commutators of weight 2 in $\widetilde{\rm N}_{n,2}$. Any palindrome in $\widetilde{\rm N}_{n,2}$ has a
normal form that is the image of the normal forms of palindromes of $\N_{n,2}$ obtained in
Proposition \ref{pnf}. We shall use the same symbol $p_1, p_2, \ldots, p_n$ to denote the normal forms in
$\widetilde{\rm N}_{n,2}$. We have
$$ p_1=y_1^{\alpha_1} z_{21}^{\alpha_1 \alpha_2} z_{31}^{\alpha_1 \alpha_3} \ldots
z_{n1}^{\alpha_1 \alpha_n},$$
$$\vdots$$
$$
p_j= y_j^{\alpha_j}
\cdot \prod_{t=1}^{j-1} ~ z_{jt}^{\alpha_j \alpha_t} \cdot \prod_{s=j+1}^n z_{sj}^{\alpha_s \alpha_j},~~~
2 \leq j \leq n-1,
$$
$$\vdots$$
$$
p_n= y_n^{\alpha_n}  \cdot  \prod_{t=1}^{n-1} z_{nt}^{\alpha_n \alpha_t}.$$

 If $w\in \widetilde{\rm N}_{n,2}$ is some element that is represented in the normal form,
 then denote by $b(w)$ the set of basis commutators of weight 2 those are in this normal form.  For example
$$b(p_1)=\{z_{21}, z_{31}, \ldots, z_{n1}\},$$
$$\vdots$$
$$b(p_j)=\{z_{j1}, \ldots, z_{j, j-1}, z_{nj}, \ldots, z_{j+1, j}\}, ~ 2 \leq j \leq n-1.$$
$$\vdots$$
$$b(p_n)=\{z_{n1}, \ldots, z_{n,n-1}\}.$$
If $w_1, \ldots, w_k \in \widetilde{\rm N}_{n,2}$ are represented in the normal form, then denote
$$b(w_1, \ldots, w_k)=\bigcup_{i=1}^k b(w_i).$$
\begin{lemma}\label{lem1}
\begin{enumerate}
\item $b(p_1, \ldots, p_n)=Bas_n$.

\medskip \item For arbitrary $i$, $1 \leq i \leq n$, $ b(p_1, \ldots, p_n)-b(p_i)=Bas_n$.

\medskip \item $b(p_1, \ldots, p_n)-(b(p_i) \cup b(p_j)) \neq Bas_n$ if $1 \leq i < j \leq n$.
\end{enumerate}
\end{lemma}
\begin{proof}
(1) follows from the fact that any basic commutator $z_{kl}, ~ 1 \leq l < k \leq n$ appeared in $p_j$, $1 \leq j \leq n$.

(2) Note that any commutator $z_{kl}$ is contained in the normal forms  $p_k$ and $p_l$. Hence if we remove $b(p_i)$ from $b(p_1, \ldots, p_n)$, then we will have all commutators of $Bas_n$.

(3) Note that $b(p_1, \ldots, p_n) - (b(p_i) \cup b(p_j))$ does not contain $z_{ji}$.
\end{proof}
\begin{lemma}\label{lem2}
The element $g=\prod_{1 \leq l < k \leq n} ~ z_{kl}$ in $\widetilde{\rm N}_{n, 2}$ can not be written as a product of less
than $2(n-1)$ palindromes.
\end{lemma}
\begin{proof}
We see that $b(g)=Bas_n$. Hence to represent $g$ as a product of palindromes, we must take at least $n-1$ different types of palindromes. Suppose that $g=q_1 \ldots q_s, ~ n-1 \leq s<2(n-1)$ is a product of $s$ palindromes. Since at least $n-1$ types of palindromes are included in this product and $s<2(n-1)$, there is a palindrome of some type  that appears only one time in the product. Without loss of generality, we can assume that it is the palindrome of type $p_1$, say, 
$$q_1=y_1 c, ~ c \in \widetilde{\rm N}_{n, 2}'.$$
Then the product $g=q_1, \ldots, q_s$ does not lie in the commutator subgroup $\widetilde{\rm N}_{n, 2}'$,
since in the normal form it contains element $y_1$. This is a contradiction.  Hence $g$ can not be written as a product of less that
$2(n-1)$ palindromes.
\end{proof}

\subsubsection{Proof of the second part of \thmref{nn2}}
\begin{proof}
We have already proved earlier that $\pw(\widetilde{\rm N}_{n, 2}) \leq 2(n-1)$. It follows from \lemref{lem2}
that there exists at least one element in $\widetilde{\rm N}_{n, 2}$ whose palindromic length is at least
$2(n-1)$. Thus $\pw(\widetilde{\rm N}_{n, 2}) \geq 2(n-1)$.\end{proof}
\subsubsection*{Proof of \corref{cor1}}
\begin{proof}
There is an onto homomorphism from $\N_{n,2}$ onto $\widetilde{\rm N}_{n, 2}$. Hence
$\pw(\widetilde{\rm N}_{n, 2}) \leq \pw(\N_{n,2})$ by \lemref{onto}. The corollary now follows from \lemref{c1}
and  \cite[Theorem 1.1]{BG}.
\end{proof}

\section{Proof of Proposition \ref{n:32}}\label{n32}

It follows from Proposition \ref{pnf} that palindromes in $\widetilde{\rm N}_{3,2}$ are of the following form:
\begin{equation}\label{eq1}
 p_{(\alpha_0, 2 \alpha_1, 2 \alpha_2)} = y_1^{\alpha_0} z_{21}^{\alpha_0 \alpha_1}   z_{31}^{\alpha_0 \alpha_2},
\end{equation}
\begin{equation}\label{eq2}
 p_{(2\alpha_1, \alpha_0, 2 \alpha_2)}  =  y_2^{\alpha_0}
 z_{21}^{\alpha_0 \alpha_1}z_{32}^{\alpha_0 \alpha_2},
\end{equation}
\begin{equation}\label{eq3}
 p_{(2 \alpha_1, 2 \alpha_2, \alpha_0)} =  y_3^{\alpha_0}
z_{31}^{\alpha_0 \alpha_1} z_{32}^{\alpha_0 \alpha_2}.
\end{equation}

\medskip In the following, for simplicity, we denote  the palindromes of the form \eqnref{eq1}, \eqnref{eq2}
and \eqnref{eq3} by $p_1$, $p_2$ and $p_3$ respectively forgetting the subscripts. When we write a product,
for eg. $p_1 p_1 p_1$, it should be understood that each $p_1$ is a palindrome of the type $\eqnref{eq1}$
but not necessarily with the same subscript unless it is mentioned otherwise. The rest of this Section will
be devoted to the proof of Proposition \ref{n:32}.

From Lemma \ref{lem2} follows lemma.
\begin{lemma}\label{z} {\it For $1 \leq j < i \leq 3$ let $z_{ij}=[y_i, y_j]$ in $\widetilde{\rm N}_{3,2}$.
The element $g=z_{21} z_{31} z_{32}$ in $\widetilde{\rm N}_{3,2}$ has palindromic length $l_{\mathcal{P}}(g)$
is equal $4$.}
\end{lemma}
\begin{proof}
From Lemma \ref{lem2} follows that $l_{\mathcal{P}}(g)\geq 4$.

On the other hand, note that
\begin{eqnarray*}
g&=& [y_2, y_1] [y_3, y_1] [y_3, y_2] \\
&=& [y_3, y_2] [y_2, y_1] [y_3, y_1] \\
&=& y_2[y_3, y_2]   y_2 [y_2, y_1]  y_1 [y_3, y_1]  y_1\\
&=& p_2 p_2 p_1 p_1
\end{eqnarray*}
Thus $g$ can be expressed as a product of four  palindromes.
\end{proof}


Note that any element $w$ of $\widetilde{\rm N}_{3,2}$ has the form
$$w=y_1^{a_1} \, y_2^{a_2} \, y_3^{a_3} z_{21}^{b_{1}} z_{31}^{b_{2}} z_{32}^{b_{3}},
$$
where, for $i=1, 2, 3$, $a_i$, $b_i \in \{0, 1\}$. Define
$$|w|=\sum_{i=1}^3 (a_i + b_i).$$
If $|w|=1$ then $\l(w) \leq 2$, since, any commutator  $z_{ij}$  is a product of two palindromes.

\medskip Let $|w|=2$, then we have 15 possibilities for $(a_1, a_2, a_3, b_1, b_2, b_3)$, where each of the $a_i$ and $b_i$ is either $0$ or $1$. For simplicity of notation we identify  the 6-tuple   $(a_1, a_2, a_3, b_1, b_2, b_3)$ with the binary word $a_1a_2 a_3 b_1 b_2 b_3$ and write  down the 15 possibilities below:

\medskip

\noindent 110000, 101000, 100100, 100010, 100001, 011000, 010100, 010010, 010001, 001100, 001010, 001001, 000110, 000101, 000011.

\medskip
 In the first twelve cases we have a product of two generators or a product of one generator and a commutator. The palindromic length of this product is $\leq 3$. In the last three cases we have:
$$000110: \ w=z_{21} z_{31} = y_2 y_1 y_2 \cdot y_3 y_1 y_3 \cdot y_1.$$
$$000101: \ w = z_{21} z_{32} = z_{32} z_{21} = y_3 y_2 y_3 \cdot y_1 y_2 y_1.$$
$$000011: \ w= z_{31} z_{32} = y_3 y_1 y_3 y_1 y_3 \cdot y_2 y_3 y_2.$$
Thus in each cases $w$ is a product of at most three palindromes.

\medskip Let $|w|=3$, then we have $\begin{pmatrix} 6 \\ 3 \end{pmatrix}=20$ possibilities:

\medskip \noindent 111000, 110100, 110010, 110001, 101100, 101010, 101001, 100110 100101, 100011, 011100, 011010, 011001, 010110, 010101, 010011, 001110, 001101, 001011, 000111.

\medskip \noindent After rearranging terms and simplification we get:
$$110010: \ w = y_1 y_2 z_{31} = z_{31} y_1 y_2= y_3 y_2 y_3 \cdot  y_2.$$
$$110001: \ w= y_1 y_2 z_{32} = y_1 z_{32} y_2 = y_1 \cdot y_3 y_2 y_3; \ \
101100: \ w = z_{21} y_1 y_3= y_2 y_1 y_2 \cdot y_3.$$
$$ 101010: \ w = y_1 y_3 z_{31} = y_3 y_1; \ \ 101001: \ w= y_1 y_3 z_{32}= y_1 \cdot y_2 y_3 y_2.$$
$$100110: \ w=y_1 z_{21} z_{31} = z_{21} y_1 z_{31}= y_2 y_1 y_2 y_1 \cdot y_1 \cdot y_3 y_1 y_3 y_1
= y_2 y_1 y_2 \cdot y_3 y_1 y_3 \cdot y_1.$$
$$100101: \ w =y_1 z_{21} z_{32} =z_ {32} z_{21} y_1=y_3 y_2 y_3 \cdot y_1 \cdot y_2. $$
$$ 100011: \ w=y_1 z_{31} z_{32}=z_{31} y_1 z_{32} = y_3 y_1 \cdot y_2 y_3 y_2.$$
$$010110: \ w= y_2 z_{21} z_{31} = y_1 y_2 y_1 \cdot y_3 y_1 y_3 \cdot y_1.  \ \
010101: \ w = y_2 z_{21} z_{32} = y_1 y_2 y_1 \cdot y_3 y_2 y_3 \cdot y_2. $$
$$010011:  \ w= z_{31} z_{32}y_2 = y_3y_1y_3 \cdot y_1 \cdot y_3 y_2 y_3. \ \ \
001110: \ w =z_{21} y_3 z_{31}=y_2 y_1 y_2 \cdot y_3 y_1. $$
$$001101: \ w=z_{21} y_3 z_{32}=y_2 y_1 y_2 \cdot y_1 \cdot y_2 y_3 y_2. \ \
001011: \ w =y_3 z_{31} z_{32} = y_1 y_3 y_1 \cdot y_3 y_2 y_3 \cdot y_2.$$
Thus we see that in each of the above cases, $w$ is a product of at most three palindromes.
Finally  $000111: \ w=z_{21} z_{31} z_{32}$ is a product of four palindromes as we have seen in Lemma \ref{z}.

\medskip Let $|w|=4$. Then we have $\begin{pmatrix} 6 \\ 4 \end{pmatrix}=15$ possibilities:

\medskip \noindent 111100, 111010, 110110, 101110, 011110, 111001, 110101, 101101, 011101, 110011, 101011, 011011, 100111, 010111, 001111.

\medskip \noindent We have after rearranging terms and simplification,
$$111100: \ w= y_1 y_2 y_3 z_{21}=y_1 y_2 z_{21} y_3 = y_1 y_2  y_2 y_1 y_2 y_1 y_3=y_2 y_1 y_3.$$
$$110110: \ w=y_1  z_{21} z_{31} = z_{31} y_1 y_2 z_{21} = y_3 y_1 y_3 \cdot y_1 y_2 y_1.$$
$$ 101110: \ w=y_1 y_3 z_{21} z_{31} =z_{21} y_1 y_3 z_{31} =  y_2 y_1 y_2 \cdot y_1 y_3 y_1.$$
$$011110: \ w= y_2  z_{21} y_3 z_{31} = y_1 (y_2 y_3) y_1;  \ \ 111001: \ w= y_1 y_2 y_3 z_{32} = y_1 y_3 y_2.$$
$$110101: \ w=y_1 y_2 z_{21} z_{32} = y_2 y_1 \cdot y_3 y_2 y_3 y_2= y_2 y_1 y_2 \cdot y_2 y_3 y_2 y_3 y_2.$$
$$  101101: \ w=z_{21} y_1 y_3 z_{32} = y_2 y_1 y_2 \cdot y_2 y_3 y_2.$$
$$011101: \ w = y_2 z_{21} y_3 z_{32} = y_1 y_2 y_1 \cdot y_2 y_3 y_2; \ \ 110011: \ w= z_{31} y_1 z_{32} y_2= y_3 y_1 y_2 y_3=y_3y_1y_3 \cdot y_3 y_2 y_3.$$
$$101011: \ w=z_{31} y_1 y_3 z_{32} = y_3 y_1 y_3 \cdot y_2 y_3 y_2; $$
$$011011:  \ w=z_{31} y_1 z_{32} z_{21} = y_3 y_1 y_3 \cdot y_3 y_2 y_3 \cdot y_1 y_2 y_1.$$
$$010111:  \ w=y_2 z_{21} z_{31} z_{32} = y_1 y_ 2 y_1 \cdot y_3 y_1 y_3 y_1 y_3 \cdot y_2 y_3 y_2.$$
$$ 001111: \ w=y_3 z_{21} z_{31} z_{32}=z_{21} y_3 z_{31} z_{32} = y_2 y_1 y_2 \cdot y_3 y_1 y_3 \cdot y_2 y_3 y_2.$$
Thus we see that in each of the above cases $w$ is a product of at most three palindromes.

\medskip Let $|w|=5$. There are  six possibilities and after rearranging terms and simplification we have:
$$111110: \ w=y_1 y_2 y_3 z_{21} z_{31}=y_1 y_2 z_{21} y_3 z_{31} = y_2 y_3 y_1.$$
$$111101: \ w=y_1 y_2 y_3 z_{21} z_{32} = y_1 y_2 z_{21} y_3 z_{32} = y_2 y_1 y_2 \cdot y_3 y_2.$$
$$111011: \ w=y_1 y_2 y_3 z_{31} z_{32}= y_1 z_{31} y_2 y_3 z_{32} = y_1 y_3 y_1 \cdot y_3 y_1 y_3 \cdot y_2.$$
$$110111: \ w=y_1 y_2 z_{21} z_{31} z_{32} = y_2 y_1 z_{31} z_{32}= z_{32} y_2 y_1 z_{31}= y_3 y_2 y_3 \cdot y_1 y_3 y_1 y_3 y_1.$$
$$101111: \ w = y_1 y_3 z_{21} z_{31} z_{32} = y_1 y_3 z_{31} z_{32} z_{21} = y_3 y_1 z_{32} z_{21} = y_3 z_{32} z_{21} y_1 = y_2 y_3 y_2\cdot y_2 y_1 y_2.$$
$$011111: \ w = y_2 y_3 z_{21} z_{31} z_{32} = z_{21} z_{31} y_2 y_3 z_{32} = y_3 y_1 x_3x_1 y_3 \cdot y_1 y_2 y_1. $$
Thus $w$ is a product of at most three palindromes.

\medskip Let $|g|=6$. Then the only possibility is $111111$ and we have
$$w=y_1 y_2 y_3 z_{21} z_{31} z_{32} = y_1 y_2 z_{21} y_3 z_{31} z_{32} = y_2 \cdot y_3 y_1 y_3 \cdot y_2 y_3 y_2.$$

\medskip Thus we have shown that all but $g=z_{21} z_{31} z_{32}$ in $\widetilde{\rm N}_{3,2}$ can be
written
as a product of at most three palindromes.  From \lemref{z} it follows that the element $g$ is the only
element whose palindromic length is  $4$.
This proves Proposition \ref{n:32}.

\section{Palindromic Width of Some Abelian-By-Nilpotent Groups}\label{abns}

\subsection{Palindromic Width of a Metabelian Group} \label{ma}

In \cite{BG}, we proved that if $AN_n$ is a free abelian-by-nilpotent group with basis $X=\{x_1, \ldots, x_n\}$,
 then $n \leq \pw(AN_n, X) \leq 5n$. To prove this we used the following representation of elements of $AN_n$
 that follows from \cite[Theorem 2]{MR}.
\begin{theorem}\label{mr1}
Let $AN_n=\langle x_1, \ldots, x_n \rangle$ be a non-abelian free abelian-by-nilpotent group of rank $n$.
Let $A$ be an abelian normal subgroup of $AN_n$ such that $AN_n/A$ is nilpotent. Then every element
$g \in AN_n$ can be expressed as:
$$g=x_1^{\alpha_1} \ldots x_n^{\alpha_n} [u_1, x_1]^{a_1} [u_2, x_2]^{a_2} \ldots [u_n, x_n]^{a_n}$$
for $u_1, \ldots, u_n \in AN_n$ and $a_1, \ldots, a_n \in A$, $\alpha_1, \ldots, \alpha_n \in \Z$.
\end{theorem}
 Evidently, every metabelian group is an abelian-by-nilpotent group. However, for finitely generated metabelian groups, this provides a better upper bound. 
\subsubsection{Proof of \thmref{math}}

\begin{proof}
Let $h \in  \gamma_2(M_n)$. Using the fact that $\gamma_2(M_n)$ is abelian, we let in
 \thmref{mr1} $A = \gamma_2(M_n)$. Hence $h$ has the form
$$h=[u_1, x_1]\ldots[u_n, x_n]~\hbox{ for}~ u_1, \ldots, u_n \in  \gamma_2(F_n(\mathcal U^2)). $$
Hence every $g \in  M_n$ has the form:
$$g=x_1^{\alpha_1} x_2^{\alpha_2} \ldots x_n^{\alpha_n}[u_1, x_1]\ldots[u_n, x_n].$$
Observe that
$$x_n^{\alpha_n}[u_n, x_n]=[u_n, x_n]^{x_n^{-\alpha_n} } x_n^{\alpha_n}=[v_n, x_n]x_n^{\alpha_n}$$
 for $v_n = x_n^{\alpha_n} u_n x_n^{-\alpha_n}$. So, $g$ can be written as
$$g=x_1^{\alpha_1} x_2 ^{\alpha_2} \ldots x_{n-1}^{\alpha_{n-1}} [v_n, x_n]x_n^{\alpha_n} [u_1, x_1]\ldots[u_{n-1}, x_{n-1}].$$
Now it follows from \lemref{l:32}(3) that
$$\pw(M_n, X) \leq (n-1)+3+3(n-1)=4n-1.$$
This proves the result.
\end{proof}

\subsection{Palindromic Width of $\widetilde{AN}_n$}
\begin{lemma}\label{cmm}
Let $G$ be a group which is generated by the set of involutions $Y = \{ y_1, y_2, \ldots, y_n \}$.
Let $g, h$ be any element in $G$ and $p$ be a palindrome in $G$. Then the following hold
\begin{enumerate}
\item{Any commutator of the type  $[g, p]$  is
a product of $2$ palindromes.
Any element $[g, y] y^{\alpha}$, $y \in Y$, $\alpha \in
\{0, 1\}$ is a product of $2$ palindromes.}

\item{Any commutator of the type  $[g, y]^h$  is
a product of $2$ palindromes.}
\end{enumerate}
\end{lemma}

\begin{proof}
(1) See that
$$
[g,  p]={g}^{-1} { p}^{-1}  g  p=  \overline{g} \, \overline{p} \, g \cdot p
$$
is a product of palindromes $\overline{g} \, \overline{p} \,  g$ and $p$.

Similarly,
$$
[g, y] y^{\alpha}=\overline{g} y g y^{1+\alpha},
$$
which is a palindrome if $\alpha = 1$ or a product of two palindromes if $\alpha = 0$.

(2) We have
$$
[g, y]^h = h^{-1} g^{-1} y g y h = \overline{h} \overline{g} y g h \cdot \overline{h} y h
$$
is a product of 2 palindromes.
\end{proof}
\subsubsection{Proof of \thmref{abn}}
\begin{proof}
$(i)$ We have $\widetilde{\N}_{n,1}=\Z_2 \times \ldots \times  \Z_2$ ($n$-times). We see that there is a homomorphism
$\widetilde{AN}_n \to \widetilde{\N}_{n,1}$. Now the left-side of the inequality follows from the fact
$\pw(\widetilde{\N}_{n,1})=n$, see \cite{BG} for a proof of this fact. To prove the right-hand side inequality, write any element $g \in \widetilde{AN}_n$ in the form
$$
g=y_1^{\epsilon_1} \ldots y_n^{\epsilon_n} [g_1, y_1]^{a_1} [g_2, y_2]^{a_2} \ldots [g_n, y_n]^{a_n},
$$
where for $i=1, \ldots, n$, $\epsilon_i \in \{0, 1\}$, $g_i\in \widetilde{AN}_n$ and $a_i \in \widetilde{A}$; here  $\widetilde{A}$
is the image of $A$ under the homomorphism $AN_n \to \widetilde{AN}_n$.
Such a representation of $g$ follows from \thmref{mr1}.
By \lemref{cmm}, for $i=1, \ldots, n$,  $[g_i, y_i]^{a_i}$ is a product of 2 palindromes. Hence $g$ is a product of at most $n+2n=3n$ palindromes. This proves the first part of the theorem.

\medskip $ (ii)$  Using the fact that $\widetilde{M}_n'$ is abelian, we set in \thmref{mr1},  $A = \widetilde{M}_n'$. Since any
two commutators of $\widetilde{M}_n$ commute, hence every $g \in  \widetilde{M}_n$ has the form:
$$
g=y_1^{\varepsilon_1} y_2^{\varepsilon_2} \ldots y_n^{\varepsilon_n} [u_1, y_1] \ldots [u_n, y_n] =
y_1^{\varepsilon_1} y_2^{\varepsilon_2} \ldots y_n^{\varepsilon_n} [u_n, y_n] [u_1, y_1] \ldots
[u_{n-1}, y_{n-1}]
$$
for some $\epsilon_i \in \{0, 1\}$, $u_i\in \widetilde{M}_n$.
Observe that
$$
y_n^{\varepsilon_n} [u_n, y_n]=[u_n, y_n]^{y_n^{-\varepsilon_n} } y_n^{\varepsilon_n}=
[v_n, y_n]y_n^{\varepsilon_n}
$$
where $v_n=y_n^{\varepsilon_n} u_n y_n^{-\varepsilon_n}$.  So, $g$ can be written as
$$g=y_1^{\varepsilon_1} y_2 ^{\varepsilon_2} \ldots y_{n-1}^{\varepsilon_{n-1}} [v_n, y_n]y_n^{\varepsilon_n} [u_1, y_1]\ldots[u_{n-1}, y_{n-1}].$$
Now it follows from \lemref{cmm} that
$$\pw(\widetilde{M}_n, Y) \leq n-1+2+2(n-1)=3n-1.$$
This proves the result.
\end{proof}

\end{document}